\documentclass[11pt,reqno]{amsart} 
\usepackage{amsmath,tikz}
\usepackage[utf8]{inputenc}
\usepackage{hyperref}
\usepackage{amsmath}
\usepackage{amssymb}
\usepackage{amsthm}
\usepackage{mathtools}
\usepackage{verbatim}
\usepackage{xcolor}
\usepackage{fullpage}
\usepackage[skip=3pt, indent=20pt]{parskip}
\usepackage{caption}
\theoremstyle{plain}
\newtheorem{thm}{Theorem}[section]
\newtheorem{lem}[thm]{Lemma}
\newtheorem{prop}[thm]{Proposition}
\newtheorem{cor}[thm]{Corollary}

\newtheorem{conj}[thm]{Conjecture}

\theoremstyle{definition}
\newtheorem{defn}[thm]{Definition}
\newtheorem{ex}[thm]{Example}

\theoremstyle{remark}

\newcommand{\newword}[1]{\emph{\color{black} {#1}}}

\newcommand{\Sym}{\mathfrak{S}}
\newcommand{\Peak}{\mathsf{Peak}}
\newcommand{\Pin}{\mathsf{Pin}}
\newcommand{\N}{\mathbb{N}}

\newcommand{\pfrk}{\mathfrak{p}}

\newcommand{\APS}{\mathsf{APS}}
\newcommand{\APSB}{\mathsf{APS}^B}
\newcommand{\APSD}{\mathsf{APS}^D}
\newcommand{\APSPLUS}{\mathsf{APS}^+}

\usepackage{tikz}

\begin{document}
\title{Pinnacle sets of signed permutations}

\date{\today}
\author{Nicolle Gonz\'{a}lez}
\address[N.~Gonz\'{a}lez]{Department of Mathematics, University of California, Berkeley, CA,  94720}
\email{\textcolor{blue}{\href{mailto:nicolle@math.berkeley.edu}{nicolle@math.berkeley.edu}}}

\author{Pamela E. Harris}
\address[P.~E.~Harris]{Department of Mathematical Sciences, University of Wisconsin, Milwaukee, WI 53211}
\email{\textcolor{blue}{\href{mailto:peharris@uwm.edu}{peharris@uwm.edu}}}

\author{Gordon Rojas Kirby}
\address[G.~Rojas Kirby]{Department of Mathematics and Statistics, San Diego State University, CA 92182}
\email{\textcolor{blue}{\href{mailto:gkirby@sdsu.edu}{gkirby@sdsu.edu}}}
 
\author{Mariana Smit Vega Garcia}
\address[M.~Smit Vega Garcia]{Department of Mathematics, Western Washington University, Bellingham, WA 98225}
\email{\textcolor{blue}{\href{mailto:}{smitvem@wwu.edu}}}

\author{Bridget Eileen Tenner}
\address[B.~E.~Tenner]{Department of Mathematical Sciences, DePaul University, Chicago, IL 60614}
\email{\textcolor{blue}{\href{mailto:bridget@math.depaul.edu}{bridget@math.depaul.edu}}}

\maketitle
\begin{abstract}
    
    Pinnacle sets record the values of the local maxima for a given family of permutations. They were introduced by Davis-Nelson-Petersen-Tenner as a dual concept to that of peaks, previously defined by Billey-Burdzy-Sagan. In recent years pinnacles and admissible pinnacles sets for the type $A$ symmetric group have been widely studied. In this article we define the pinnacle set of signed permutations of types $B$ and $D$. We give a closed formula for the number of type $B$/$D$ admissible pinnacle sets and answer several other related enumerative questions. 
\end{abstract}

\section{Introduction}

The study of permutation statistics is an active subdiscipline of combinatorics. Given a permutation $w=w(1)w(2)\cdots w(n)$, two particularly well-studied statistics are \newword{descents} and \newword{peaks}. Respectively, these statistics refer to indices $i$ such that $w(i)>w(i+1)$, and indices $i$ such that $w(i-1)<w(i)>w(i+1)$. The collection of a permutation's descent indices is its descent set, with a permutation's peak set being similarly defined. Two fundamental goals in the study of these particular statistics are (1) understanding which subsets can arise as descent sets or peak sets (i.e., which sets are \newword{admissible} as descent or peak sets), and (2) enumerating the permutations that have a given admissible descent or peak set. 

For descents of permutations in the (type $A$) symmetric group $\Sym_n$, this question was answered by Stanley \cite[Ex. 2.2.4]{Stanley} and is well known to give rise to the Eulerian numbers. Inspired by Stembridge's study of peaks in the context of poset partitions \cite{Stembridge}, Billey, Burdzy, and Sagan \cite{BBS} introduced the study of \newword{admissible peak sets} for $\Sym_n$ with an interest in probabilistic applications, and established that the number of permutations with peak set $I$ is given by $2^{n-|I|-1}p(n)$, where $p(n)$ is a polynomial of degree $\max(I)-1$. 
Shortly thereafter, their results were extended to permutations in type $B$ by Castro-Velez et al.~\cite{CDOPZ} where it was shown that the number of permutations with a given peak set $I$ is $ 2^{2n-|I|-1}p(n)$, with $p(n)$ the same as in \cite{BBS} above. 
The second author and various collaborators went further by extending these results to types $C$ and $D$ \cite{DHIP}, using peaks to study properties of the descent polynomial \cite{DHIOS}, 
and then initiating the study of peaks in the context of graphs \cite{DEHIMO}. 

A notion that is closely related to peaks is the \newword{pinnacle set} of a permutation. \newword{Pinnacles} are the set of values held by the permutation at the peak indices. More precisely, given a permutation $w=w(1)w(2)\cdots w(n)$ with peak set $\Peak(w)$, the pinnacle set of $w$ is $\Pin(w)=\{w(i):\mbox{$i \in \Peak(w)$}\}$.
Given a subset $I\subseteq[n]$, if there exists a permutation $w$ whose pinnacle set is $I$, we say that $I$ is an \newword{admissible pinnacle set}. In \cite{PinnaclesTypeA}, Davis, Nelson, Petersen, and the last author pioneered the study of pinnacles for permutations in $\Sym_n$ and gave a complete characterization of admissible pinnacle sets. They provided a closed formula for the number of admissible pinnacle sets with a given maximum value, as well as a refinement to those appearing in $\Sym_n$. In particular, Davis et al.~gave a recursive formula for the number of permutations in $\Sym_n$ with a given pinnacle set $p(n)$ and asked whether a more efficient expression could be computed. This paper led to a sequence of articles in recent years, many focused on improved and faster formulas for $p(n)$, by realizing permutations with given pinnacle sets as invariants under certain modified $\Sym_n$-actions \cite{DHHIN,FNT} or via more traditional enumerative methods \cite{DLMSSS, fang, Minnich}. In related work, Rusu \cite{rusu} and Rusu-Tenner \cite{rusu-tenner} deepened the knowledge of pinnacles in $\Sym_n$ by investigating further properties of these statistics and characterizing \newword{admissible pinnacle orderings}. 

In this article we look beyond type $A$ and study pinnacles and admissible pinnacle sets 
for the type $B$ and type $D$ \newword{signed symmetric groups}, $\Sym_n^B$ and $\Sym_n^D$. Our main results are the following, where we write \newword{$\APS_n^X$} to denote the admissible pinnacle sets in $\Sym_n^X$ for $X \in \{A,B,D\}$:

\begin{enumerate}
    
    \item Theorem~\ref{T:countB} gives a closed formula for the number of admissible pinnacle sets in $\Sym_n^B$,
    $$|\APSB_n| = \sum_{k=0}^{\left\lfloor\frac{n-1}{2}\right\rfloor}\binom{n}{k} \binom{n-1-k}{\left\lfloor\frac{n-1}{2}\right\rfloor-k}.$$
   
   \item Theorem~\ref{T:2kBD} proves that any admissible pinnacle set in $\Sym_{2k}^B$ is also admissible in $\Sym_{2k}^D$; that is, $\APSD_{2k} = \APSB_{2k}$. 
    
    \item In counterpoint to Theorem~\ref{T:2kBD}, Theorem~\ref{T:BnotD} counts the admissible pinnacle sets of type $B$ that are not in type $D$ when $n=2k+1$, 
    $$\vert \APSB_{2k + 1} \setminus \APSD_{2k + 1} \vert = \binom{2k-1}{k}.$$
    
    \item Theorems~\ref{T:plusnotA1} and~\ref{T:plusnotA2} count the all-positive admissible pinnacle sets of type $B$ that are not admissible in type $A$. Namely, defining \newword{$\APSPLUS_n$} $\coloneqq \{ S \in \APSB_n : S \subset \mathbb{N}\};$ we prove that the sets $\APSPLUS_n\setminus \APS_n$ are enumerated by, 
    $$\left|\APSPLUS_n\setminus \APS_n\right|=
\begin{cases}
4^k - \binom{2k}{k}&\mbox{if $n = 2k + 1$, and}\\
2^{2k-1} - \binom{2k}{k}& \mbox{if $n=2k$.}
\end{cases}
$$ 
\end{enumerate}

This article is organized as follows. 
In Section~\ref{sec:background}, we introduce all the necessary background and notation, defining pinnacles and related notions in type $B$.
In Section~\ref{sec:admissible signed pinnacles}, we give a characterization of admissible signed pinnacle sets and a formula for their enumeration. In Section~\ref{sec: relating A, B, and D}, we provide relations between admissible pinnacle sets of type $A$, $B$, and $D$. Lastly, in Section~\ref{sec:future}, we describe some future directions and open conjectures.

\subsection*{Acknowledgements}
The authors thank Patrek K\'arason Ragnarsson for the coding and data that facilitated the research in this project, and Freyja K\'arad\'ottir Ragnarsson for the key insight to the proof of Theorem~\ref{T:BnotD}. 
The authors also thank the American Institute of Mathematics and the National Science Foundation for sponsoring the Latinx Mathematicians Research Community, which brought together a subset of the authors initially for collaboration.

P.E.H. was partially supported through a Karen Uhlenbeck EDGE Fellowship, M.S.V.G was partially supported by the NSF grant DMS 2054282, and B.E.T. was partially supported by the NSF grant DMS-2054436.

\section{Background}\label{sec:background}

Let $\N=\{1,2,3,\ldots\}$ and for $n\in\N$ we write $[n] \coloneqq\{1,2,\ldots,n\}$. For any set $X$, typically of positive values, although we make the definition more generally, we define
$$-X \coloneqq \{-x : x \in X\}.$$
Finally, we define
$$\pm X = X \cup -X.$$
Throughout this paper, we let $\Sym_n$ denote the (type $A$) symmetric group. That is, $\Sym_n$ is the group of bijections from $[n] \to [n]$ under function composition. We often write $w \in \Sym_n$ using one-line notation, as $w=w(1) w(2)\cdots w(n)$.

The type $B$ symmetric group (that is, the hyperoctahedral group) is the group of \newword{signed permutations} $\Sym^B_n$. These are bijections $\pm[n] \to \pm[n]$ such that 
\[w(-i) = -w(i) \text{ for all $i \in [n]$}.\]
In particular, any $w \in \Sym_n^B$ satisfies the property that $\{ |w(1)|,\ldots,|w(n)| \}=[n]$.

The type $D$ symmetric group is the subgroup $\Sym_n^D$ of $\Sym_n^B$ consisting of \newword{signed permutations with an even number of signs}. Namely, these are the signed permutations $w$ for which 
\[|\{i \in [n] \ : \ w(i)<0\}| \text{ is even. } \]

As in type $A$, we use one-line notation to denote signed permutations $w \in \Sym_n^B$, where we may write only $w=w(1) w(2)\cdots w(n)$ since this uniquely determines $w(-i)$ for all positive $i$. Following convention, we write $-i = \bar{i}$ to ease notation. For example, $w = \bar{1}2\bar{3}$ is the signed permutation with $w(1) = -1$, $w(2) = 2$, and $w(3) = -3$. 

Recall that a permutation $w\in \Sym_n$ has a \newword{peak} at index $i\in\{2,\ldots,n-1\}$ if 
\[w(i-1)<w(i)>w(i+1),\]
and the value $w(i)$ is a \newword{pinnacle} of $w$. We denote by $\Peak(w)$ the set of all peaks of $w\in\Sym_n$. 
The \newword{pinnacle set} of $w\in\Sym_n$ is 
\[
\Pin(w) = \{w(i) \ : \ i\in \Peak(w)\}.
\]

\begin{defn}
A set $P\subseteq[n]$ is an \newword{$n$-admissible pinnacle set} in type $A$ if there exists a permutation $w\in\Sym_n$ such that $\Pin(w)=P$, and we call the permutation $w$ a \newword{witness} for the set $P$. 
\end{defn}

For example, the identity permutation is a witness for the admissible pinnacle set $\emptyset$ (as is any peak-less permutation). Denote by $\APS_n$ the set of all $n$-admissible pinnacle sets in type $A$.  
Analogously, in later sections, we will write $\APSB_n$ for the admissible pinnacle sets in type $B$ and $\APSD_n$ for the admissible pinnacle sets in type $D$.

In order to facilitate our discussions about pinnacles, we introduce terminology about their minimal counterparts: a permutation $w\in\Sym_n$ has a \newword{valley} at index $i\in \{2,\ldots,n-1\}$ if 
\[w(i-1)>w(i)<w(i+1),\]
and the value $w(i)$ is a \newword{vale} of $w$.

\begin{figure}[ht]
\center
\begin{tikzpicture}[yscale=.5, xscale=.75]
\draw[thick] (0.5,8.5)--(0.5,0.5)--(8.5,0.5);
\node at (1,0)[scale=.75]{$1$};
\node at (2,0)[scale=.75]{$2$};
\node at (3,0)[scale=.75]{$3$};
\node at (4,0)[scale=.75]{$4$};
\node at (5,0)[scale=.75]{$5$};
\node at (6,0)[scale=.75]{$6$};
\node at (7,0)[scale=.75]{$7$};
\node at (8,0)[scale=.75]{$8$};
\node at (0,1)[scale=.75]{$1$};
\node at (0,2)[scale=.75]{$2$};
\node at (0,3)[scale=.75]{$3$};
\node at (0,4)[scale=.75]{$4$};
\node at (0,5)[scale=.75]{$5$};
\node at (0,6)[scale=.75]{$6$};
\node at (0,7)[scale=.75]{$7$};
\node at (0,8)[scale=.75]{$8$};
\draw[thick] (1,2)--(2,3)--(3,7)--(4,1)--(5,5)--(6,6)--(7,4)--(8,8);
\node at (1,2){$\bullet$};
\node at (2,3){$\bullet$};
\node at (3,7){$\bullet$};
\node at (4,1){$\bullet$};
\node at (5,5){$\bullet$};
\node at (6,6){$\bullet$};
\node at (7,4){$\bullet$};
\node at (8,8){$\bullet$};
\draw[red, thick] (3,7) circle (10pt);
\draw[red, thick] (6,6) circle (10pt);
\draw[blue] (4,1) circle (8pt);
\draw[blue] (7,4) circle (8pt);
\end{tikzpicture} 
\caption{The graph of the permutation $23715648 \in \Sym_8$ with the pinnacles/peaks circled in red and the vales/valleys in blue.\label{fig:typeApinnacles}}
\end{figure}
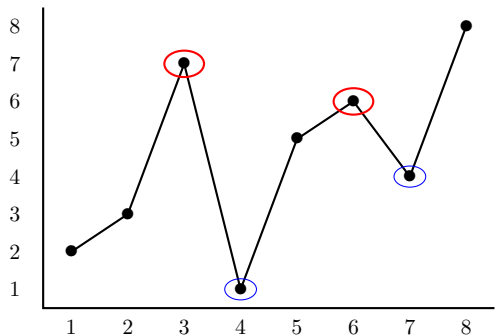

\begin{ex}
Consider the permutation $w= 23715648 \in \Sym_8$ shown in Figure~\ref{fig:typeApinnacles}. We have $\Peak(w) = \{ 3,6 \}$ and $\Pin(w) = \{ 6, 7\}$, and valleys and vales 
$\{4,7\}$ and $\{ 1,4\}$, respectively.
\end{ex}

\subsection{Pinnacles in types \texorpdfstring{$B$}{B} and \texorpdfstring{$D$}{D}}

Pinnacles were defined in \cite{PinnaclesTypeA} for unsigned permutations, but they could just as easily have been defined for signed permutations---or, in fact, for arbitrary strings of distinct real numbers. We now expand the type $A$ definitions to type $B$, and note that since $\Sym_n^D\subset \Sym_n^B$, these definitions also hold for type $D$.

\begin{defn}
Let $w$ be a signed permutation. A \newword{pinnacle} of $w$ is a value $w(i)$ that is larger than both $w(i-1)$ and $w(i+1)$. The \newword{pinnacle set} of $w$ is the set of its pinnacles. We will write $\Pin^B(w)$ to denote the pinnacle set of a signed permutation $w$.
\end{defn}

In order to define admissible pinnacle sets, it is important to establish which subsets
could even appear among the one-line notation of a signed permutation.

\begin{defn}
A \newword{signed set} (or \newword{signed subset}, depending on context)
is a set $I$ such that $x \in I$ implies $-x \not\in I$.
\end{defn}

Throughout this paper, we assume that all subsets of $\pm [n]$ are signed subsets.

\begin{defn}
A signed subset $S \subset \pm [n]$ is an \newword{admissible pinnacle set} if $S$ is the pinnacle set of some signed permutation. That permutation is a \newword{witness} for $S$.
\end{defn}

Note that when we study sets that are admissible as pinnacle sets in type $D$, any witness permutation will be required to be in $\Sym_n^D$ for some $n$. 
As before, we denote by \newword{$\APSB_n$} (resp., \newword{$\APSD_n$}) the set of all $n$-admissible pinnacle sets in type $B$ (resp., type $D$).
Once again, we have $\emptyset \in \APSD_n \subseteq \APSB_n$. For example, $123\cdots {n}$ and $\bar{2}\bar{1}34\cdots{n}$ are both witnesses for $\emptyset$.

While there can be multiple witness permutations for a given admissible pinnacle set, we will often refer to a particular witness permutation that we call ``canonical.''

\begin{defn}\label{defn:canonical witness}
For $S \in \APSB_n$, write $S = \{s_1 < s_2 < \cdots < s_k\}$, and set
$$S' \coloneqq -[n] \setminus \{-|s| : s \in S\} = \{s_1' < s_2' < \cdots < s_{n-k}'\}.$$
Then the \newword{canonical witness permutation} is 
$$w \coloneqq s_1' \ s_1 \ s_2' \ s_2 \ \cdots \ s_k' \ s_k \ s_{k+1}' \ \cdots \ s_{n-k}' \in \Sym_n^B.$$
If $S \in \APSD_n$, then its canonical (type $D$) witness permutation is $w$ as defined above if $w$ is in $\Sym_n^D$, and otherwise its canonical witness is obtained from $w$ by replacing $s_{n-k}'$ with $|s_{n-k}'|$.

The canonical witness permutation of a set $S \in \APS_n$ is defined similarly, but with $S' := [n] \setminus S$. This produces an unsigned permutation $w \in \Sym_n$ having pinnacle set $S$.
\end{defn}

Next we establish that the ``canonical witness permutations'' are, in fact, witnesses and follow this by providing canonical witness permutations in Example~\ref{ex:b but not d}. 

\begin{lem}
The canonical witness permutation for an admissible set $S$ is a witness for $S$.
\end{lem}

\begin{proof}
The set $S$ is admissible, so there is some permutation whose pinnacle set is $S$. The canonical witness has been constructed to have minimal possible non-pinnacle values, and to position the smallest non-pinnacle values beside the smallest pinnacle values. Therefore, if any permutations were to have $S$ as a pinnacle set (and we know that some permutation does), the permutation given in Definition~\ref{defn:canonical witness} would be among them.
\end{proof}

Although $\Sym_n^B$ contains both $\Sym_n$ and $\Sym_n^D$ as subgroups, there are interesting subtleties to the pinnacle sets that become admissible when witness permutations can be signed. First, some sets will be admissible with type $B$ permutations, but not with type $D$ permutations. And second, some sets of all-positive values will be admissible with type $B$ permutations, but not with type $A$ (unsigned) permutations. We demonstrate each of these scenarios below.

\begin{figure}[htbp]
\center
\begin{tikzpicture}[yscale=.5, xscale=.75]
\draw (-1,8) node {(a)};
\draw[black!30, dotted, ultra thick] (.5,0) -- (7.5,0);
\draw[thick] (0.5,7.5)--(0.5,-7.5)--(7.5,-7.5);
\node at (1,-8)[scale=.75]{$1$};
\node at (2,-8)[scale=.75]{$2$};
\node at (3,-8)[scale=.75]{$3$};
\node at (4,-8)[scale=.75]{$4$};
\node at (5,-8)[scale=.75]{$5$};
\node at (6,-8)[scale=.75]{$6$};
\node at (7,-8)[scale=.75]{$7$};
\node at (0,1)[scale=.75]{$1$};
\node at (0,2)[scale=.75]{$2$};
\node at (0,3)[scale=.75]{$3$};
\node at (0,4)[scale=.75]{$4$};
\node at (0,5)[scale=.75]{$5$};
\node at (0,6)[scale=.75]{$6$};
\node at (0,7)[scale=.75]{$7$};
\node at (0,0)[scale=.75]{$0$};
\node at (0,-1)[scale=.75]{$-1$};
\node at (0,-2)[scale=.75]{$-2$};
\node at (0,-3)[scale=.75]{$-3$};
\node at (0,-4)[scale=.75]{$-4$};
\node at (0,-5)[scale=.75]{$-5$};
\node at (0,-6)[scale=.75]{$-6$};
\node at (0,-7)[scale=.75]{$-7$};
\draw[thick] (1,-7)--(2,-4)--(3,-6)--(4,1)--(5,-5)--(6,2)--(7,-3);
\node at (1,-7){$\bullet$};
\node at (2,-4){$\bullet$};
\node at (3,-6){$\bullet$};
\node at (4,1){$\bullet$};
\node at (5,-5){$\bullet$};
\node at (6,2){$\bullet$};
\node at (7,-3){$\bullet$};
%
\draw[red, thick] (2,-4) circle (10pt);
\draw[red, thick] (4,1) circle (10pt);
\draw[red, thick] (6,2) circle (8pt);
\draw[blue] (3,-6) circle (8pt);
\draw[blue] (5,-5) circle (8pt);
\end{tikzpicture} 
\hspace{.5in}
\begin{tikzpicture}[yscale=.5, xscale=.75]
\draw (-1,8) node {(b)};
\draw[white] (8,8) node {(b)}; 
\draw[black!30, dotted, ultra thick] (.5,0) -- (7.5,0);
\draw[thick] (0.5,7.5)--(0.5,-7.5)--(7.5,-7.5);
\node at (1,-8)[scale=.75]{$1$};
\node at (2,-8)[scale=.75]{$2$};
\node at (3,-8)[scale=.75]{$3$};
\node at (4,-8)[scale=.75]{$4$};
\node at (5,-8)[scale=.75]{$5$};
\node at (6,-8)[scale=.75]{$6$};
\node at (7,-8)[scale=.75]{$7$};
\node at (0,1)[scale=.75]{$1$};
\node at (0,2)[scale=.75]{$2$};
\node at (0,3)[scale=.75]{$3$};
\node at (0,4)[scale=.75]{$4$};
\node at (0,5)[scale=.75]{$5$};
\node at (0,6)[scale=.75]{$6$};
\node at (0,7)[scale=.75]{$7$};
\node at (0,0)[scale=.75]{$0$};
\node at (0,-1)[scale=.75]{$-1$};
\node at (0,-2)[scale=.75]{$-2$};
\node at (0,-3)[scale=.75]{$-3$};
\node at (0,-4)[scale=.75]{$-4$};
\node at (0,-5)[scale=.75]{$-5$};
\node at (0,-6)[scale=.75]{$-6$};
\node at (0,-7)[scale=.75]{$-7$};
\draw[thick] (1,-6)--(2,3)--(3,-5)--(4,4)--(5,-1)--(6,7)--(7,-2);
\node at (1,-6){$\bullet$};
\node at (2,3){$\bullet$};
\node at (3,-5){$\bullet$};
\node at (4,4){$\bullet$};
\node at (5,-1){$\bullet$};
\node at (6,7){$\bullet$};
\node at (7,-2){$\bullet$};
%
\draw[red, thick] (2,3) circle (10pt);
\draw[red, thick] (4,4) circle (10pt);
\draw[red, thick] (6,7) circle (8pt);
\draw[blue] (3,-5) circle (8pt);
\draw[blue] (5,-1) circle (8pt);
\end{tikzpicture} 
\caption{(a) The graph of the permutation $\bar{7}\bar{4} \bar{6} 1 \bar{5}2 \bar{3} \in \Sym_7^B$ with the pinnacles/peaks circled in red and the vales/valleys in blue. (b) The graph of the permutation $\bar{6} 3 \bar{5} 4 \bar{1} 7 \bar{2} \in \Sym_7^B$ with the pinnacles/peaks circled in red and the vales/valleys in blue.}\label{fig:typeBpinnacles}\label{fig:all positive b but not in a}
\end{figure}
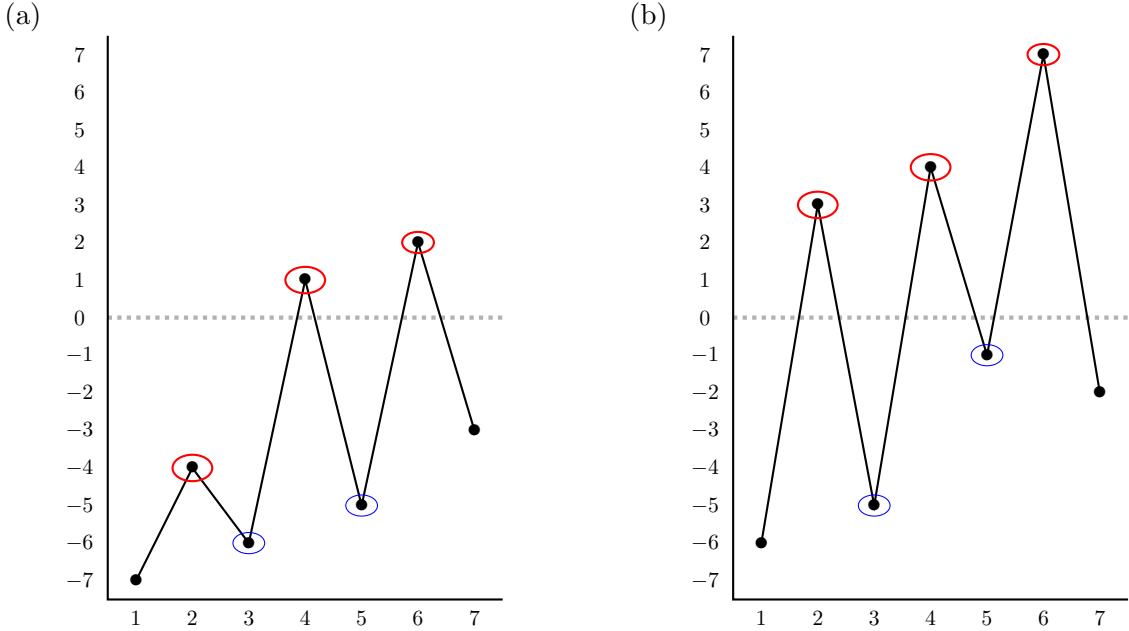

\begin{ex}\label{ex:b but not d}\label{ex:all positive b but not a}
The set $\{\bar{4},1,2\}$ is admissible in $\Sym_7^B$, with canonical witness 
permutation $\bar{7}\bar{4}\bar{6}1\bar{5}2\bar{3}$ as shown in Figure~\ref{fig:typeBpinnacles}(a). However, there is no element of $\Sym_7^D$ having this pinnacle set. 
That is, $\{\bar{4},1,2\} \not\in \APSD_7$. The set $\{3,4,7\}$ is admissible in $\Sym_7^B$, with canonical witness permutation $\bar{6} 3 \bar{5} 4 \bar{1} 7 \bar{2}$, as shown in Figure~\ref{fig:all positive b but not in a}(b). However, despite its pinnacle set having all positive values, there is no type $A$ permutation having this pinnacle set. That is, $\{3,4,7\} \not\in \APS_n$ for any $n$.
\end{ex}

\section{Admissible signed pinnacle sets in type \texorpdfstring{$B$}{B}}\label{sec:admissible signed pinnacles}

In this section, we  characterize and enumerate the admissible pinnacle sets among signed permutations. This  expands on the work begun in \cite{PinnaclesTypeA} for unsigned permutations, but, as we show, the results for signed permutations are subtly different from those in type $A$.

\subsection{Characterization of admissible pinnacle sets}

For the remainder of the article, we will often use the fact that given an admissible pinnacle set $S\in\APSB_n$, we can always write
$$S = P(S) \sqcup N(S)$$
with 
\[P(S) \coloneqq S \cap [n] \text{ and } N(S) \coloneqq S \cap -[n].\]
When no confusion will arise, we simply write $P \coloneqq P(S)$ and $N\coloneqq N(S)$.

To give a first inkling of how admissible pinnacle sets in type $B$ are fundamentally different from those in type $A$, we note that there are some sets of positive integers that are never in $\APS_n$ for any $n$. For example, any set containing $1$ or $2$ will never be the pinnacle set of any permutation in $\Sym_n$. On the other hand, such a statement is not true in type $B$.

\begin{lem}\label{lem:AnyIntegersAreAPS}
Every finite signed subset $S$ is admissible in $\Sym_n^B$, for some $n \in \N$. That is, there exists $w\in \Sym^B_n$ such that $S = \Pin^B(w)$. 
\end{lem}

\begin{proof}
Write $S=\{s_1< \cdots < s_k\}$. Let $m = \max\{|s| : s \in S\}$ (that is, $m = \max\{|s_1|,|s_k|\}$). Define the set $S' \coloneqq -[2m+1] \setminus \{-|s| : s \in S\}$, which we write as $S' = \{s_1' < \cdots < s_{2m+1-k}'\}$. Then
$$w = s_1' \ s_1 \ s_2' \ s_2 \ \cdots \ s_k' \ s_k \ s_{k+1}' \ s_{k+2}' \ \cdots \ s_{2m+1-k}' \in \Sym_{2m+1}^B,$$
and $\Pin^B(w) = S$.
\end{proof}

Using a similar argument as the one proving Lemma~\ref{lem:AnyIntegersAreAPS}, it follows that any finite set of all positive values is admissible in some $\Sym_n^B$. 

\begin{cor}\label{cor:any all positive set is admissible in type b}
Any subset $P \subset [n]$ with $|P| \le \frac{n-1}{2}$ is admissible in $\Sym_n^B$.
\end{cor}

\begin{proof}
Let $P=\{p_1 < \cdots < p_k\}$, and set $P'\coloneqq-([n]\setminus P)=\{p_1'<\cdots<p_{n-k}'\}$. Then 
$$w=p_1' \  p_1 \ p_2' \ p_2 \ \cdots \ p_k' \ p_k \ p_{k+1}' \ p_{k+2}' \ \cdots \ p_{n-k}'\in \Sym_n^B$$
and $\Pin^B(w)=P$.
\end{proof}

This can be particularly interesting when the set $P$ was not admissible in $\Sym_n$.

\begin{ex}
Consider $P = \{1,2\}$ with $n = 5$. The permutation $\bar{5}1\bar{4}2\bar{3} \in \Sym^B_5$ is a witness 
permutation for $P$, so $P \in \APSB_5$, while $P \not\in \APS_n$ for any $n$.
\end{ex}

Our goal is to establish a characterization and formula for the number of admissible pinnacle sets in $\Sym_n^B$. We begin with some preliminary steps, from which those results will follow. The first of these is a bijection between admissible pinnacle sets in $\Sym_n$ and those admissible pinnacle sets in $\Sym_n^B$ that have no positive values.

\begin{lem}\label{lem:positive aps in bijective with negative aps}
There exists a bijection between $\APS_n$ and $\{ S \in \APS^B_n : S \subset -\mathbb{N}\}$.
\end{lem}

\begin{proof}
Given $T\in \APS_n$, define $T' \coloneqq \{t - (n+1) : t \in T\}$. The set $T'$ has no positive elements. Let $w \in \Sym_n$ be the canonical witness for $T$. Then $w' \coloneqq (w(1)-(n+1)) \cdots (w(n) - (n+1)) \in \Sym_n^B$ has pinnacle set $T'$, and so $T' \in \APSB_n$. 

This process can be inverted: given $S\in\APSB_n$ with $P(S) = \emptyset$, map this $S$ to $S' \coloneqq \{s + n+1 : s \in S\}$. It follows that $S'\in\APS_n$, as before.
\end{proof}

We illustrate Lemma~\ref{lem:positive aps in bijective with negative aps} with an example.

\begin{ex}
The set $\{3,6,7,10\} \in \APS_{10}$ is in correspondence with $\{\bar{8}, \bar{5}, \bar{4}, \bar{1}\} \in \APSB_{10}$. The permutations described in the proof of Lemma~\ref{lem:positive aps in bijective with negative aps}, which exhibit these sets as pinnacle sets, are shown in Figure~\ref{fig:bijection}.
\end{ex}

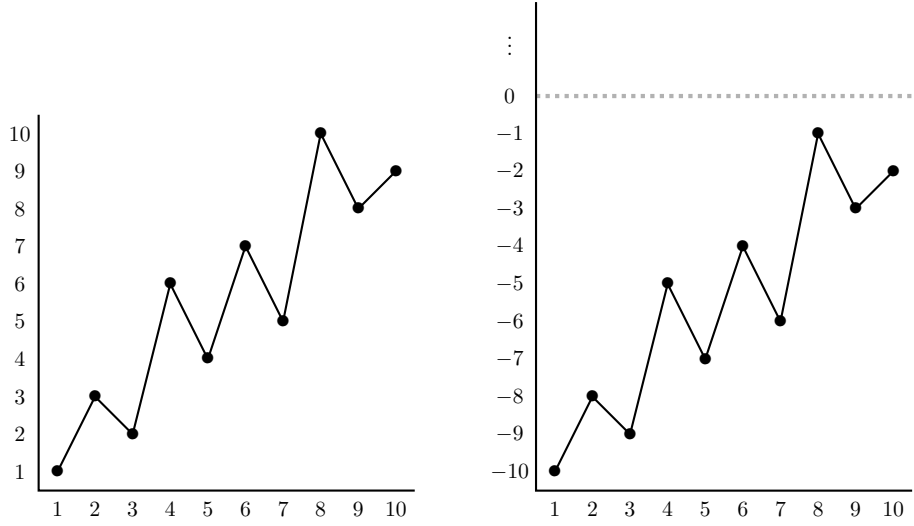
\begin{figure}[htbp]
\center
\begin{tikzpicture}[yscale=.5, xscale=.5]
\draw[thick] (0.5,10.5)--(0.5,.5)--(10.5,.5);
\node at (1,0)[scale=.75]{$1$};
\node at (2,0)[scale=.75]{$2$};
\node at (3,0)[scale=.75]{$3$};
\node at (4,0)[scale=.75]{$4$};
\node at (5,0)[scale=.75]{$5$};
\node at (6,0)[scale=.75]{$6$};
\node at (7,0)[scale=.75]{$7$};
\node at (8,0)[scale=.75]{$8$};
\node at (9,0)[scale=.75]{$9$};
\node at (10,0)[scale=.75]{$10$};
\node at (0,1)[scale=.75]{$1$};
\node at (0,2)[scale=.75]{$2$};
\node at (0,3)[scale=.75]{$3$};
\node at (0,4)[scale=.75]{$4$};
\node at (0,5)[scale=.75]{$5$};
\node at (0,6)[scale=.75]{$6$};
\node at (0,7)[scale=.75]{$7$};
\node at (0,8)[scale=.75]{$8$};
\node at (0,9)[scale=.75]{$9$};
\node at (0,10)[scale=.75]{$10$};
\node(a1) at (1,1){$\bullet$};
\node(a2) at (2,3){$\bullet$};
\node(a3) at (3,2){$\bullet$};
\node(a4) at (4,6){$\bullet$};
\node(a5) at (5,4){$\bullet$};
\node(a6) at (6,7){$\bullet$};
\node(a7) at (7,5){$\bullet$};
\node(a8) at (8,10){$\bullet$};
\node(a9) at (9,8){$\bullet$};
\node(a10) at (10,9){$\bullet$};
\draw[thick] (1,1)--(2,3)--(3,2)--(4,6)--(5,4)--(6,7)--(7,5)--(8,10)--(9,8)--(10,9);
\end{tikzpicture} 
\qquad
\begin{tikzpicture}[yscale=.5, xscale=.5]

\draw[black!30, dotted, ultra thick] (.5,0) -- (10.5,0);
\draw[thick] (0.5,2.5)--(0.5,-10.5)--(10.5,-10.5);
\node at (1,-11)[scale=.75]{$1$};
\node at (2,-11)[scale=.75]{$2$};
\node at (3,-11)[scale=.75]{$3$};
\node at (4,-11)[scale=.75]{$4$};
\node at (5,-11)[scale=.75]{$5$};
\node at (6,-11)[scale=.75]{$6$};
\node at (7,-11)[scale=.75]{$7$};
\node at (8,-11)[scale=.75]{$8$};
\node at (9,-11)[scale=.75]{$9$};
\node at (10,-11)[scale=.75]{$10$};
\node at (-.2,-1)[scale=.75]{$-1$};
\node at (-.2,-2)[scale=.75]{$-2$};
\node at (-.2,-3)[scale=.75]{$-3$};
\node at (-.2,-4)[scale=.75]{$-4$};
\node at (-.2,-5)[scale=.75]{$-5$};
\node at (-.2,-6)[scale=.75]{$-6$};
\node at (-.2,-7)[scale=.75]{$-7$};
\node at (-.2,-8)[scale=.75]{$-8$};
\node at (-.2,-9)[scale=.75]{$-9$};
\node at (-.2,-10)[scale=.75]{$-10$};
\node at (-.2,0)[scale=.75]{$0$};
\node at (-.2,1.5)[scale=.75]{$\vdots$};
\node at (1,-10){$\bullet$};
\node at (2,-8){$\bullet$};
\node(a3) at (3,-9){$\bullet$};
\node(a4) at (4,-5){$\bullet$};
\node(a5) at (5,-7){$\bullet$};
\node(a6) at (6,-4){$\bullet$};
\node(a7) at (7,-6){$\bullet$};
\node(a8) at (8,-1){$\bullet$};
\node(a9) at (9,-3){$\bullet$};
\node(a10) at (10,-2){$\bullet$};
\draw[thick] (1,-10)--(2,-8)--(3,-9)--(4,-5)--(5,-7)--(6,-4)--(7,-6)--(8,-1)--(9,-3)--(10,-2);

\end{tikzpicture} 
\caption{The left-hand figure shows the canonical witness for $\{3,6,7,10\}$ in $\Sym_{10}$. The right-hand figure shows the corresponding witness permutation for $\{\bar{8}, \bar{5}, \bar{4}, \bar{1}\}$, as defined in the proof of Lemma~\ref{lem:positive aps in bijective with negative aps}. } \label{fig:bijection}
\end{figure}

We have defined admissible pinnacle sets in types $A$, $B$, and $D$, referring to permutations in $\Sym_n$, $\Sym^B_n$, or $\Sym^D_n$. However, as suggested earlier, there is a natural generalization of admissible pinnacle sets to permutations of \emph{any} totally ordered set.

\begin{defn}\label{def:pinnacles X}
For any totally ordered set $X$, let $\APS(X)$ be the set of \newword{admissible pinnacle sets} of $X$. The definitions of \newword{witness} and \newword{canonical witness} permutations in this general setting are analogous to their definitions in the symmetric groups.
\end{defn}

\begin{ex}The set $X = \{-2, \pi, 4, 5, 100\}$ has six admissible pinnacle sets:
$$\emptyset,\ \{4\},\ \{5\},\ \{100\},\ \{4, 100\}, \text{ and } \{5,100\}.$$ 
\end{ex}

Note that if we are only interested in how many admissible pinnacle sets $X$ has, as opposed to the sets themselves, then the size of $X$ is what matters. 

\begin{lem}\label{lem:aps of generic set}
For any totally ordered finite set $X$, $|\APS(X)| = |\APS\left([|X|]\right)| = |\APS_{|X|}|$.
\end{lem}

This calculation will be useful in the enumeration appearing in the next subsection.

We are now are able to characterize admissible pinnacle sets for signed permutations. 

\begin{thm}\label{thm:characterizing aps in type b}
The sets in $\APSB_n$ are exactly the sets $S = P(S) \sqcup N(S)$ for which
\begin{itemize}
\item $|P(S)| + |N(S)| \le (n-1)/2$,
\item $P(S) \subset [n]$,
\item $N(S) \subset - ([n] \setminus P(S))$, and 
\item $N(S) \in \APS(-([n]\setminus P(S)))$. 
\end{itemize}
\end{thm}

\begin{proof}
First of all, it is clear that any admissible pinnacle set in $\Sym_n^B$ must satisfy the listed requirements.

Now suppose that a set $S$ satisfies the listed requirements, with $P \coloneqq P(S) = \{p_1 < \cdots < p_k\}$ and $N \coloneqq N(S) = \{n_1 < \cdots < n_r\}$. In light of the last requirement, let $w$ be the canonical witness permutation of the set $(-([n]\setminus P))$, having pinnacle set $N$.
That is,
$$w = i_1 \ n_1 \ i_2 \ n_2 \ \cdots \ i_r \ n_r \ i_{r+1} \ i_{r+2} \ i_{r+3} \ \cdots \ i_{n-k-r}$$
where $i_j < i_{j+1}$ and $\{i_1, \ldots, i_{n-k-r}\} = -([n]\setminus P) \setminus N$.  Then
$$i_1 \ n_1 \ i_2 \ n_2 \ \cdots \ i_r \ n_r \ i_{r+1} \ p_1 \ i_{r+2} \ p_2 \ i_{r+3} \ \cdots \ p_k \ i_{r+k+1} \ i_{r+k+2} \ \cdots \ i_{n-r-k}$$
is a canonical witness for $S = P \sqcup N$ in $\Sym_n^B$. Hence $S \in \APSB_n$.
\end{proof}

\subsection{Enumeration of admissible pinnacle sets}

The conditions listed in Theorem~\ref{thm:characterizing aps in type b} inform our enumeration of the admissible pinnacle sets in $\Sym_n^B$. In particular, we will construct these sets by first fixing a collection $P$ of positive pinnacles and then determining how many sets $N$ of negative pinnacles exist for which $P \cup N$ is admissible in $\Sym_n^B$. 

In order not to have too many pinnacles (that is, not more than $\lfloor (n-1)/2 \rfloor$), we need to understand the following value. 

\begin{defn}
Let $\pfrk_n(d)$ be the number of admissible pinnacle sets in $\Sym_n$ having cardinality at most $d$. That is,   
$$\pfrk_n(d)\coloneqq|\{S\in \APS_n \ : \ |S|\le d\}|.$$
\end{defn}

This statistic has a particularly nice formula.

\begin{prop}\label{prop:size of pfrk}
For all integers $d \in \left[0,\left\lfloor (n-1)/2\right\rfloor\right]$,
$$\pfrk_n(d) = \binom{n-1}{d}.$$
\end{prop}

\begin{proof}
 
The admissible pinnacle sets in $\Sym_n$ having cardinality at most $d$ can be partitioned into two sets: those that contain $n$, and those that do not. We claim that the first set is counted by $\pfrk_{n-1}(d-1)$, and the second set is counted by $\pfrk_{n-1}(d)$. 

Suppose, first, that $S \in \APS_n$ such that $n \in S$ and $|S| = k \le d$. Let $w \in \Sym_n$ be the canonical witness for $S$. Deleting $n$ from the one-line notation of $w$ will produce a permutation $v \in \Sym_{n-1}$ with $\Pin(v) = S \setminus \{n\}$. Conversely, given $T \in \APS_{n-1}$ with $|T| = k-1$, let $u \in \Sym_{n-1}$ be the canonical witness for $T$. Inserting $n$ between the non-pinnacles $u(2k-1)$ and $u(2k)$ will produce a permutation in $\Sym_n$ whose pinnacle set is $T \cup \{n\}$. This establishes the first part of the claim.

For the second part of the claim, suppose that $S \in \APS_n$ with $n \not\in S$ and $|S| = k \le d$. Let $w \in \Sym_n$ be the canonical witness for $S$. Because $n \not\in S$, we have $w(n) = n$. Thus the permutation $w(1) \cdots w(n-1) \in \Sym_{n-1}$ has pinnacle set $S$. Conversely, if $v \in \Sym_{n-1}$ has pinnacle set $S$, then appending $n$ to the end of $v$ will produce a permutation in $\Sym_n$ that also has pinnacle set $S$.

This gives the binomial recurrence
$$\pfrk_n(d) = \pfrk_{n-1}(d-1) + \pfrk_{n-1}(d).$$
To complete the argument, notice that $\pfrk_n(0) = 1$ and $\pfrk_n(1) = 1 + (n-2) = n-1$, for all positive integers $n\ge 2$.
\end{proof}

Combining Theorem \ref{thm:characterizing aps in type b}, which characterizes admissible pinnacle sets for signed permutations, with the enumeration in Proposition \ref{prop:size of pfrk}, we now count the admissible pinnacle sets for signed permutations.

\begin{thm}\label{T:countB}
If $n\geq 2$, then
$$\left|\APSB_n\right| = \sum_{k=0}^{\left\lfloor\frac{n-1}{2}\right\rfloor}\binom{n}{k} \binom{n-1-k}{\left\lfloor\frac{n-1}{2}\right\rfloor-k}.$$
\end{thm}

\begin{proof}
The main idea of the proof will be to construct admissible pinnacle sets in $\Sym_n^B$ following the requirements of Theorem~\ref{thm:characterizing aps in type b}. First, we will select a set $P$ of positive pinnacles. In other words, $P \subset [n]$ and $|P| \le (n-1)/2$. Then we add to it any set $N \subset -([n]\setminus P)$ that is in $\APS(-([n]\setminus P))$, so long as $|P| + |N| \le (n-1)/2$. We are interested in the number of such sets,  and Lemma~\ref{lem:aps of generic set} says that we only need to care about the size of $P$ in this process. This and Lemma~\ref{lem:positive aps in bijective with negative aps} mean that such sets $N$ can be counted in terms of admissible pinnacle sets of $\Sym_{n-|P|}$.

Fix an integer $k \in [0, (n-1)/2]$, and choose a $k$-element subset $P \subset [n]$. There are $\binom{n}{k}$ ways to do this. We can supplement $P$ with any $r$-element admissible pinnacle set $N \subset -([n]\setminus P)$, as long as $k+r \le \lfloor(n-1)/2\rfloor$. The number of ways to do this is
$$\pfrk_{n-k}\left(\left\lfloor\frac{n-1}{2}\right\rfloor-k\right).$$
Therefore, by Proposition~\ref{prop:size of pfrk}, the number of admissible pinnacle sets in $\Sym_n^B$ is
$$\sum_{k=0}^{\left\lfloor\frac{n-1}{2}\right\rfloor}\binom{n}{k} \binom{n-1-k}{\left\lfloor\frac{n-1}{2}\right\rfloor-k},$$
as desired.
\end{proof}

In Table~\ref{table:counting aps in type b}, we give the number of signed admissible pinnacle sets in type $B$ for $3 \le n\leq 15$, while permutations in $\Sym_1^B$ and $\Sym_2^B$ have no pinnacles. This appears in the OEIS as sequence \cite[\href{http://oeis.org/A359066}{A359066}]{OEIS}.  The even-indexed terms in the table appear in \cite[\href{http://oeis.org/A240721}{A240721}]{OEIS} and the odd-indexed terms appear in \cite[\href{http://oeis.org/A178792}{A178792}]{OEIS}. 

\begin{table}[htbp]
{\renewcommand{\arraystretch}{2}$\begin{array}{c||c|c|c|c|c|c|c|c|c|c|c|c|c}
n & 3 & 4 & 5 & 6 & 7 & 8 & 9 & 10 & 11 & 12 & 13 & 14 & 15\\
\hline
\hline
\left\vert \APSB_n \right\vert & 5 & 7 & 31 & 49 & 209 & 351 & 1471 & 2561 & 10625 & 18943 & 78079 & 141569 & 580865
\end{array}$}
\vspace{.1in}
\caption{The number of admissible pinnacle sets in $\Sym_n^B$, for $3 \le n \le 15$.}\label{table:counting aps in type b}
\end{table}

In the next Section, we will be able to answer the analogous enumerative question in type D (see Corollary \ref{C:D}).

\section{Relating admissible pinnacle sets in types  \texorpdfstring{$A$, $B$, and $D$}{A, B, D}}\label{sec: relating A, B, and D}

There is a natural embedding of $\Sym_n$ in $\Sym_n^D$, and of $\Sym_n^D$ in $\Sym_n^B$. Having spent Section~\ref{sec:admissible signed pinnacles} analyzing pinnacle sets that are admissible in $\Sym_n^B$, it is natural to wonder how these sets are related to those that are admissible in $\Sym_n^D$ or, for those elements of $\APSB_n$ without negative values, to those that are admissible in $\Sym_n$. We now give complete characterization of each of these relationships. 

\subsection{Comparing admissible pinnacle sets in types \texorpdfstring{$B$ and $D$}{B and D}}
As mentioned before, $\Sym_n^D \subset \Sym_n^B$, thus it is natural to investigate the relationship between those sets that are admissible as pinnacle sets in type $B$ and those that are in type $D$. It is, perhaps, not surprising that this relationship depends on the parity of $n$.

As a first step in this analysis, we identify a technique that will be handy in proving that a set is admissible for type $D$. 

\begin{lem}\label{lem:flipping last letter to get another witness}
Suppose that $w \in \Sym_n^B$ is a witness for a pinnacle set $S$. If $w(n-1) > \pm w(n)$ or if $w(n-1) < \pm w(n)$, then the permutation $w'$, defined by
$$w'(i) = \begin{cases} w(i) & i < n \text{ and}\\ -w(i) & i = n,\end{cases}$$
is also a witness for $S$. Moreover, $S \in \APSD_n$.
\end{lem}

\begin{proof}
First note that $w'$ is an element of $\Sym_n^B$ because changing the sign of the last letter does not alter the fact that this is a signed permutation on $\pm[n]$. Next observe that the pinnacle set has not changed from $w$ to $w'$ because none of the inequalities between consecutive letters has been altered. Finally, note that the numbers of negative values in $w$ and in $w'$ differ by $1$, meaning that one of these permutations is in $\Sym_n^D$ while the other is in $\Sym_n^B \setminus \Sym_n^D$.
\end{proof}

We will call on the previous result often throughout our arguments in this section, beginning with a description of the simple relationship between $\APSB_n$ and $\APSD_n$.

\begin{thm}\label{T:2kBD}
For $k\geq 1$,
$\APSB_{2k} = \APSD_{2k}$. 
\end{thm}

\begin{proof} 
Certainly anything admissible in type $D$ is also admissible in type $B$, because signed permutations include the signed permutations in type $D$. It remains to show that any pinnacle set that is admissible in $\Sym_{2k}^B$ is also admissible in $\Sym_{2k}^D$. 
Fix $S\coloneqq \{s_1<\cdots < s_l\} \in \APSB_{2k}$. Because $l \le \left\lfloor(2k-1)/2\right\rfloor$, we have $l\le k-1$. Then the canonical witness $w$ for $S$ satisfies the hypotheses of Lemma~\ref{lem:flipping last letter to get another witness}, and so in fact $S \in \APSD_{2k}$. 
\end{proof}

The equality shown in Theorem~\ref{T:2kBD} relies on the fact that there are always at least two more non-pinnacles than there are pinnacles in signed permutations on $2k$ letters. This not necessarily true for signed permutations of an odd number of letters, and hence it is not surprising that the relationship between $\APSB_{2k+1}$ and $\APSD_{2k+1}$ has more nuance than the relationship presented in Theorem~\ref{T:2kBD}. Indeed, we will show that $\APSD_{2k+1}$ is a strict subset of $\APSB_{2k+1}$, and we will describe the elements of the latter that are not elements of the former.

\begin{lem}\label{lem:APSBnotDSize}
If $S \in \APSB_{2k+1} \setminus \APSD_{2k+1}$, then $|S|=k$.
\end{lem}

\begin{proof}
Fix $S \in \APSB_{2k+1}$ and let $w \in \Sym_{2k+1}^B$ be the canonical witness for $S$. If $|S| < k$, then both $w(2k)$ and $w(2k+1)$ are non-pinnacles and $w(2k) < w(2k+1) < 0$. In particular, the hypotheses of Lemma~\ref{lem:flipping last letter to get another witness} are satisfied by $w$, and so
$S \in \APSD_{2k+1}$. Hence, if $S \in \APSB_{2k+1} \setminus \APSD_{2k+1}$, then $|S| = k$. 
\end{proof}

One implication of Lemma~\ref{lem:APSBnotDSize} is that if $w \in \Sym_{2k+1}^B$ is a witness for $S \in \APSB_{2k+1}\setminus \APSD_{2k+1}$, then $w(3), w(5), \ldots, w(2k-1)$ are all vales. With Lemma~\ref{lem:APSBnotDSize} providing a first step toward understanding elements of $\APSB_{2k+1} \setminus \APSD_{2k+1}$, we now proceed to describe these sets more clearly.

\begin{lem}\label{lem:all non-pinnacles are always negative}
Fix $S \in \APSB_{2k+1} \setminus \APSD_{2k+1}$. In every witness permutation for $S$, the non-pinnacle values are all negative.
\end{lem}

\begin{proof}
Fix $S\in\APSB_{2k+1}\setminus \APSD_{2k+1}$ and $w\in \Sym_{2k+1}^B$ a witness for $S$. Following Lemma~\ref{lem:APSBnotDSize}, the non-pinnacles of $w$ are precisely $w(1), w(3), \ldots, w(2k+1)$. In particular, each $w(2i+1)$ is less than its immediate neighbors. Suppose, for the purpose of obtaining a contradiction, that $w(2j+1) > 0$ for some $j$. Let $w' \in \Sym_{2k+1}^B$ be the permutation obtained from $w$ by replacing $w(2j+1)$ by $-w(2j+1)$. Then $w'$ is still a witness for $S$. Either $w$ or $w'$ is in $\Sym_{2k+1}^D$, meaning that $S$ must be an element of $\APSD_{2k+1}$. This is a contradiction, so there is no such $j$. 
\end{proof}

In fact, the negative values of $S \in \APSB_{2k+1} \setminus \APSD_{2k+1}$ are enough to determine all of $S$. 

\begin{lem}\label{lem:NegPinnaclesDetermine}
Suppose that $S\in \APSB_{2k+1}\setminus \APSD_{2k+1}$, with $P \coloneqq S \cap \mathbb{N}$ and $N \coloneqq S \cap -\mathbb{N}$. Then the elements of $P$ are the smallest $k - |N|$ values in the set $[2k+1]\setminus -N$. In particular, $N$ determines $P$, and hence all of $S$.
\end{lem}

\begin{proof}
Fix $S\in \APSB_{2k+1}\setminus \APSD_{2k+1}$, with $P$ and $N$ as defined. By Lemma~\ref{lem:APSBnotDSize}, we have $|S|=k$, so let $ S= \{s_1 < s_2 < \cdots <s_k\} $.  If $|N| = k$, then there is nothing to check, so assume that $|N| < k$ and hence $s_k > 0$. 
Suppose, for the purpose of obtaining a contradiction, that there exists $q\in ([2k+1]\setminus-N)\setminus P$ with $q < s_k$. Let $q$ be minimal with these properties. Let $w$ be the canonical witness permutation for $S$. By definition, $w(2k) = s_k$ and $w(2k+1) = -q$. But then $w'$, which agrees with $w$ everywhere except $w'(2k+1) = q$, is also a witness for $S$, contradicting Lemma~\ref{lem:all non-pinnacles are always negative}. 
Therefore $P$ consists precisely of the smallest $k-|N|$ values in the set $[2k+1]\setminus -N$. 
\end{proof}

Lemma~\ref{lem:NegPinnaclesDetermine} gives a necessary condition for elements of $\APSB_{2k+1}\setminus \APSD_{2k+1}$. The next result establishes that the set $N \sqcup P$ constructed in Lemma~\ref{lem:NegPinnaclesDetermine} is, in fact, an admissible signed pinnacle set. 

\begin{cor}\label{cor:GoodNegPinnaclesAlwaysWork}
Suppose that $N \subset -\mathbb{N}$ and $N \in \APSB_{2k+1}$. Let $P$ be the smallest $k - |N|$ values in $[2k+1] \setminus -N$. Then $N \sqcup P \in \APSB_{2k+1}$. 
\end{cor}

\begin{proof}
This follows from Theorem~\ref{thm:characterizing aps in type b}.
\end{proof}

Maintaining the terminology of Corollary~\ref{cor:GoodNegPinnaclesAlwaysWork}, note that for any set $N \subset -\mathbb{N}$, all witness permutations for $N \sqcup P$ are forced by construction of $P$ to have the same number of negative values: $k+1 + |N|$. This yields the following corollary. 

\begin{cor}\label{cor:S and N sizes have same parity}
Suppose $S\in\APSB_{2k+1}\setminus \APSD_{2k+1}$, with $N \coloneqq S \cap -\mathbb{N}$. The sets $|N|$ and $|S|$ have the same parity.
\end{cor}

\begin{proof}
To have $S\in\APSB_{2k+1}\setminus \APSD_{2k+1}$, we need $|S| = k$, by Lemma \ref{lem:APSBnotDSize}. Moreover, as discussed above, the number of negative values is $k+1+|N|$, and this must be odd because $S\notin \APS_{2k+1}^D$. Thus $k + |N|=|S|+|N|$ is even, completing the proof. 
\end{proof}

The consequence of this collection of results is that if we have a set $N \subset -\mathbb{N}$ that is, itself, admissible in $\Sym_{2k+1}^B$, and for which $|N|$ has the same parity as $k$, then there is a unique ($(k-|N|)$-element) set $P \subset \mathbb{N}$ for which
$$N \sqcup P \in \APSB_{2k+1}\setminus \APSD_{2k+1}.$$
Therefore, to enumerate $\APSB_{2k + 1} \setminus \APSD_{2k + 1}$, it suffices to count the elements of $\APSB_{2k+1}$ that have no positive values and that have size of the form $k-2i$. 

Because we want to look at the elements of $\APSB_{2k+1}$ having no positive values, we can take advantage of Lemma~\ref{lem:positive aps in bijective with negative aps} to look, instead, at $\APS_{2k+1}$. That is, it will suffice to count
$$\sum_{i \ge 0}  \bigg|\{S \in \APS_{2k+1} : |S| = k-2i\}\bigg|.$$

The last step of this enumeration requires a parity result.

\begin{lem}\label{L:oddeven}
For $k\geq 0$,
$\bigg| \{S \in \APS_{2k+1} : |S| \text{ is even}\}\bigg| =  \bigg|\{S \in \APS_{2k+1} : |S| \text{ is odd}\}\bigg|$. 
\end{lem}

\begin{proof}
Fix $S\subset [2k+1]$. If $2k+1\in S$, then set $S'\coloneqq S\setminus\{2k+1\}.$ Clearly if $S\in\APS_{2k+1}$ then also $S'\in\APS_{2k+1}$, and the sets $|S|$ and $|S'|$ have different parities.

Now consider $S \in \APS_{2k+1}$ with $2k+1\not\in S$. By \cite[Theorem~1.8]{PinnaclesTypeA}, $\max(S)>2|S|$. We have $\max(S)<2k+1$, so $|S|<k$. Consequently, the canonical witness permutation $w$ of $S$ uses at most $k$ vales, so there are at least $(2k+1)-(k-1+k)=2$ non-pinnacle/non-vale values in $w$ and one of these is $2k+1$. We can create a new permutation $w'$ by inserting $2k+1$ immediately to the right of the largest vale in $w$. Thus the pinnacle set of $w'$ is $S\cup\{2k+1\}$.

Therefore there is a bijection between even-sized elements of $\APS_{2k+1}$ and odd-sized ones, obtained by adding/removing the element $2k+1$. This partitions $\APS_{2k+1}$ into two evenly sized parts. 
\end{proof}

We have now completed all of the steps necessary to give the desired enumeration.

\begin{thm}\label{T:BnotD}
For $k\geq 1$, $\displaystyle{\left\vert \APSB_{2k + 1} \setminus \APSD_{2k + 1} \right\vert = \binom{2k-1}{k}}.$
\end{thm}

\begin{proof}
Following Lemmas~\ref{lem:APSBnotDSize} and~\ref{lem:NegPinnaclesDetermine} and Corollary~\ref{cor:S and N sizes have same parity}, we can enumerate $\APSB_{2k+1}\setminus \APSD_{2k+1}$ by counting elements of $\APS_{2k+1}$ that have size $\{k-2i : i = 0, 1, \ldots\}$. These are either all of the odd-sized sets in $\APS_{2k+1}$ or all of the even-sized ones. By Lemma~\ref{L:oddeven}, then,
$$\left\vert \APSB_{2k+1}\setminus \APSD_{2k+1}\right\vert = \frac{1}{2} \left\vert \APS_{2k+1}\right\vert.$$
It was shown in \cite[Theorem~1.8]{PinnaclesTypeA} that $|\APS_{2k+1}|=\binom{2k}{k}$. Finally, it is straightforward to check that $\frac{1}{2}\binom{2k}{k}=\binom{2k-1}{k}$. 
\end{proof}

We can now use Theorem~\ref{T:countB}, which enumerated $\APSB_n$, and Theorems~\ref{T:2kBD} and~\ref{T:BnotD} to enumerate $\APSD_n$ for all $n$. 

\begin{cor}\label{C:D}
For $k\geq1$, $\left\vert\APSD_{2k}\right\vert=\left\vert\APSB_{2k}\right\vert$ and 
\begin{align*}
\left\vert\APSD_{2k+1}\right\vert &= \left(\sum_{i=0}^k \binom{2k+1}{i} \binom{2k-i}{k-i}\right) - \binom{2k-1}{k}.
\end{align*}
\end{cor}

In Table~\ref{table:counting aps in type d}, we give the number of signed admissible pinnacle sets in type $D$ for $3 \le n\leq 15$, while permutations in $\Sym_1^D$ and $\Sym_2^D$ have no pinnacles. This appears in the OEIS as sequence \href{http://oeis.org/A359067}{A359067}. The even-indexed terms are identical to even terms in Table~\ref{table:counting aps in type b} and the odd-indexed terms are $\binom{2k-1}{k}$ less than the corresponding odd-indexed terms in Table~\ref{table:counting aps in type b}.
\begin{table}[htbp]
{\renewcommand{\arraystretch}{2}$\begin{array}{c||c|c|c|c|c|c|c|c|c|c|c|c|c}
n & 3 & 4 & 5 & 6 & 7 & 8 & 9 & 10 & 11 & 12 & 13 & 14 & 15\\
\hline
\hline
\left\vert \APSD_n \right\vert & 4  & 7 & 28 & 49 & 199 & 351 & 1436 & 2561 & 10499 & 18943 & 77617 & 141569 & 579149
\end{array}$}
\vspace{.1in}
\caption{The number of admissible pinnacle sets in $\Sym_n^D$, for $3 \le n \le 15$.}\label{table:counting aps in type d}
\end{table}

\subsection{Comparing admissible pinnacle sets in types \texorpdfstring{$B$ and $A$}{B and A}} 
Some elements of $\APSB_n$ have no negative values, and so one could ask if those sets might also be admissible in $\Sym_n$. In this section we consider how those elements of $\APSB_n$ are related to the admissible pinnacle sets in $\APS_n$. To make this discussion precise, we introduce:
$$\APSPLUS_n \coloneqq \{ S \in \APSB_n : S \subset \mathbb{N}\};$$
in other word, $\APSPLUS_n$ consists of the pinnacle sets that are admissible in $\Sym_n^B$ and that contain no negative values.

For example, $\{1,3\} \in \APSPLUS_5$, with canonical witness $\overline{5}1\overline{4}3\overline{2} \in \Sym_5^B$. In fact, by Corollary~\ref{cor:any all positive set is admissible in type b}, any subset of $[n]$ having at most $(n-1)/2$ elements is admissible in $\Sym_n^B$. Contrast this with $\APS_n$; for example,
$$\APSPLUS_5 \setminus \APS_5 = \Big\{\{1\}, \{2\}, \{1,2\}, \{1,3\}, \{1,4\}, \{1,5\}, \{2,3\}, \{2,4\}, \{2,5\}, \{3,4\}\Big\}.$$

Our goal in this section is to understand $\APSPLUS_n \setminus \APS_n$. As with the comparison of $\APSB_n$ and $\APSD_n$, this will depend on the parity of $n$.

\begin{thm}\label{T:plusnotA1}
For $k\geq 0$, $|\APSPLUS_{2k+1} \setminus \APS_{2k+1}|=4^k - \binom{2k}{k}$.
\end{thm}

\begin{proof}
Because $\APS_{2k+1} \subset \APSPLUS_{2k+1}$, the desired value is equal to
$$\left\vert\APSPLUS_{2k+1}\right\vert - \left\vert\APS_{2k+1}\right\vert.$$
Following Corollary~\ref{cor:any all positive set is admissible in type b}, we can compute $\left\vert\APSPLUS_{2k+1}\right\vert$ by counting $i$-element subsets of $[2k+1]$ for all $i\le k$. The result follows by recognizing that this yields a sum that is half of a row-sum of Pascal's triangle, and combining this with the enumeration of $\APS_{2k+1}$ from \cite{PinnaclesTypeA}:
\begin{align*}
|\APSPLUS_{2k+1} \setminus \APS_{2k+1}| &= \left\vert\APSPLUS_{2k+1}\right\vert - \left\vert\APS_{2k+1}\right\vert\\
&= \binom{2k+1}{0}+ \binom{2k+1}{1}+ \cdots+ \binom{2k+1}{k} - \binom{2k}{k}\\
&= \frac{1}{2} 2^{2k+1} - \binom{2k}{k}\\
&= 4^k - \binom{2k}{k}.\qedhere
\end{align*}
\end{proof}

We now complete this analysis by considering the even-indexed case.

\begin{thm}\label{T:plusnotA2}
For $k\geq 1$, $\left\vert\APSPLUS_{2k}\setminus \APS_{2k}\right\vert =2^{2k-1} - \binom{2k}{k}.$ 
\end{thm}

\begin{proof}
This calculation is almost identical to that from the proof of Theorem~\ref{T:plusnotA1}, except that we will also have to subtract the central term from a row of Pascal's triangle:
\begin{align*}
|\APSPLUS_{2k} \setminus \APS_{2k}| &= \left\vert\APSPLUS_{2k}\right\vert - \left\vert\APS_{2k}\right\vert\\
&= \binom{2k}{0}+ \binom{2k}{1}+ \cdots+ \binom{2k}{k-1} - \binom{2k-1}{k-1}\\
&= \frac{1}{2} \left(2^{2k} - \binom{2k}{k} \right) - \binom{2k-1}{k-1}\\
&= 2^{2k-1} - \left(\frac{1}{2}\binom{2k}{k} + \binom{2k-1}{k-1}\right)\\
&= 2^{2k-1} - \binom{2k}{k}.\qedhere
\end{align*}
\end{proof}

We combine the enumerations of Theorems \ref{T:plusnotA1} and \ref{T:plusnotA2} in Table~\ref{table:counting aps+ not in type a}.  Specifically, we list $\left\vert\APSPLUS_{n}\setminus \APS_{n}\right\vert$ for $3 \leq n\leq 15$, while permutations in $\Sym_1^B$ and $\Sym_2^B$ have no pinnacles. The $n$th term of this appears in the OEIS as double the $(n-1)$st term of \cite[\href{http://oeis.org/A294175}{A294175}]{OEIS}. Moreover, the odd-indexed terms, enumerated in Theorem~\ref{T:plusnotA1}, appear in \cite[\href{http://oeis.org/A068551}{A068551}]{OEIS} and the even-indexed terms are double the terms of \cite[\href{http://oeis.org/A008549}{A008549}]{OEIS}. 

\begin{table}[ht!]
{\renewcommand{\arraystretch}{2}$\begin{array}{c||c|c|c|c|c|c|c|c|c|c|c|c|c}
n & 3 & 4 & 5 & 6 & 7 & 8 & 9 & 10 & 11 & 12 & 13 & 14 & 15\\
\hline
\hline
\left\vert \APSPLUS_n\setminus\APS_n \right\vert & 2  & 2 & 10 & 12 & 44 & 58 & 186 & 260 & 772 & 1124 & 3172 & 4760 & 12952
\end{array}$}
\vspace{.1in}
\caption{The number of all-positive pinnacle sets that are admissible in $\Sym_n^B$ but not in $\Sym_n$, for $3 \le n \le 15$.}

\label{table:counting aps+ not in type a}
\end{table}

\section{Future directions}\label{sec:future}

As demonstrated by the results in this paper, admissible pinnacle sets have rich structure and properties even beyond the symmetric group. There are many directions for further research on this topic, including broad questions about pinnacle sets for families of permutations with certain properties, and enumerative specializations. 

As a complement to those large questions, we conclude this work by pointing out that we uncovered a possible connection between $\left|\APSB_n\right|$ and sequence \cite[\href{http://oeis.org/A119258}{A119258}]{OEIS}. In particular, we have the following conjecture.

\begin{conj}\label{C:A119258} Consider the sequence \cite[\href{http://oeis.org/A119258}{A119258}]{OEIS}, given by $T(n,0) = T(n,n) = 1$ and $T(n,k) = 2T(n-1, k-1) + T(n-1,k)$ for $0<k<n.$ Then 
$$\left|\APSB_n\right| = T\left(n,\left\lfloor \frac{n-1}{2} \right\rfloor \right).$$
\end{conj}

\appendix
\section{Data}

Patrek Ragnarsson's code for computing the data in Tables~\ref{table:counting aps in type b}, \ref{table:counting aps in type d}, and~\ref{table:counting aps+ not in type a} can be found at \url{https://github.com/PatrekR/Signed-pinnacle-sets}. Note that the data in Table~\ref{table:counting aps in type d} is the difference between the enumerations given in two of the files posted at this GitHub link.

\nocite{*}
\bibliographystyle{plain}
\bibliography{Bibliography.bib}

\end{document}